\documentclass[11pt,a4paper]{article}

\usepackage{epsf,epsfig,amsfonts,amsgen,amsmath,amstext,amsbsy,amsopn,amsthm
}
\usepackage{ebezier,eepic}
\usepackage{color}
\usepackage{multirow}
\newtheorem{algorithm}{Algorithm}[section]
\newtheorem{dfn}{Definition} [section]
\newtheorem{theorem}[dfn]{Theorem}
\newtheorem{lemma}[dfn]{Lemma}
\newtheorem{problem}[dfn]{Problem}
\newtheorem{proposition}[dfn]{Proposition}

\newtheorem{corollary}[dfn]{Corollary}
\newtheorem{conjecture}[dfn]{Conjecture}

\newcommand{\CA}{{\mathcal{A}}}

\newcommand{\CF}{{\mathcal{F}}}
\newcommand{\CH}{{\mathcal{H}}}

\newcommand{\CM}{{\mathcal{M}}}
\newcommand{\CN}{{\mathcal{N}}}
\newcommand{\CP}{{\mathcal{P}}}
\newcommand{\CQ}{{\mathcal{Q}}}
\newcommand{\CC}{{\mathcal{C}}}

\newcommand{\CV}{{\mathcal{V}}}
\newcommand{\CE}{{\mathcal{E}}}

\newcommand{\CG}{{\mathcal{G}}}
\newcommand{\CB}{{\mathcal{B}}}
\def\ce#1{\lceil#1\rceil}

\begin{document}

\title{Cycles of given lengths in hypergraphs}

\date{September 26, 2016}

\author{Tao Jiang\footnote{Department of Mathematics, Miami University, Oxford, OH 45056, USA. E-mail: jiangt@miamioh.edu. Research partially supported by National Science Foundation grant DMS-1400249. Research partially carried out during the author's visit of University of Science and Technology of China, whose hospitality is gratefully acknowledged.}
~~~~~~~~
Jie Ma\footnote{School of Mathematical Science,
University of Science and Technology of China, Hefei, 230026,
P.R. China. Email: jiema@ustc.edu.cn. Research partially supported by NSFC projects 11501539 and 11622110.}}

\maketitle

\begin{abstract}
In this paper, we develop a method for studying cycle lengths in hypergraphs. Our method is built on
earlier ones used in \cite{GL-3uniform,GL,FO}.   However, instead of utilizing the well-known lemma
of Bondy and Simonovits \cite{BS74} that most existing methods do, we develop a new and very simple lemma in its place.
One useful feature of the new lemma is its adaptiveness for the hypergraph setting.

Using this new method, we prove a conjecture of Verstra\"ete \cite{V16} that for $r\ge 3$, every $r$-uniform hypergraph with average degree
$\Omega(k^{r-1})$ contains Berge cycles of $k$ consecutive lengths. This is sharp up to the constant factor.
As a key step and a result of independent interest, we prove that every $r$-uniform linear hypergraph with average degree at least $7r(k+1)$ contains
Berge cycles of $k$ consecutive lengths. 

In both of these results, we have additional control on the lengths of the cycles,
which therefore also gives us bounds on the Tur\'an numbers of Berge cycles
(for even and odd cycles simultaneously).
In relation to our main results, we obtain further improvements on the Tur\'an numbers of Berge cycles and
the Zarankiewicz numbers of even cycles.
We will also discuss some potential further applications of our method.
\end{abstract}

%\medskip
%\noindent {\bf Keywords:}

%\medskip
%\noindent {\bf Mathematics Subject Classification (2010):} 05C15; 05C45

\section{Introduction}
The study of cycles is one of the essential ingredients of Graph Theory.
In the present paper, we are mainly concerned with extremal problems on cycle lengths in graphs and hypergraphs.
One such type of problems consider the set of cycle lengths in graphs under certain density conditions
(see, for instance, the work of Sudakov and Verstraete \cite{SV08} for some in-depth discussions).
Another type of problems, which are most pertinent to this paper, consider the longest possible consecutive sequence of cycle lengths.

This can be traced back to a question of Erd\H{o}s and now a theorem of Bondy and Vince \cite{BV98} that
any graph with minimum degree at least three contains two cycles whose lengths differ by at most two.
H\"aggkvist and Scott \cite{HS98} extended this by showing that every graph with minimum degree
$\Omega(k^2)$ contains $k$ cycles of consecutive even lengths.
This quadratic bound was first improved to a linear one by Verstra\"ete in \cite{V00}, who proved that
average degree at least $8k$ will ensure the existence of $k$ cycles of consecutive even lengths.
Since then there has been an extensive research \cite{Fan,SV08,KSV,Ma} on related topics.
Very recently Liu and the second author \cite{LM} proved a tight result that every graph $G$ with minimum degree at least $2k+1$
contains $k$ cycles of consecutive even lengths;
and if $G$ is 2-connected and non-bipartite, then $G$ also contains $k$ cycles of consecutive odd lengths.
Among others, one closely related problem is the study of cycle lengths modulo a fixed integer $k$.
This was proposed by Burr and Erd\H{o}s \cite{Erd76} forty years ago and
some conjectures were formalized by Thomassen \cite{Th83} in 1983
(we refer interested readers to \cite{SV16} for a thorough introduction).

Our work is also closely related to the so-called {\it Tur\'an problem}.
Let $\mathcal{F}$ be a family of  $r$-graphs. An $r$-graph is {\it $\mathcal{F}$-free}
if it does not contain any member of $\mathcal{F}$ as a subhypergraph.
The {\it Tur\'an number} $ex_r(n,\mathcal{F})$ of the family $\mathcal{F}$ denotes the
maximum number of hyperedges contained in an $n$-vertex $\mathcal{F}$-free $r$-graph.
When $r=2$, we will write it as $ex(n,\mathcal{F})$. Studying the Tur\'an function $ex(n,\CF)$ 
for graphs and hypergraphs has been a central problem in extremal graph theory ever since the work of
P. Tur\'an \cite{turan}. For  non-bipartite graphs, the problem is asymptotically solved by the celebrated Erd\H{o}s-Stone-Simonovits Theorem 
\cite{Es-Stone} (see also \cite{ES}).
However, the Tur\'an problem for bipartite graphs remains mostly open, with the special case for  even cycles $C_{2k}$
receiving particular attention.
A classic theorem of Bondy and Simonovits \cite{BS74} shows that $ex(n,C_{2k})\le 100k\cdot n^{1+1/k}$.
This bound was subsequently improved by several authors in \cite{V00,Pik,BJ}. The Tur\'an problem for cycles in hypergraphs
has been investigated for different notions of hypergraph cycles in \cite{BG, FJ, FO, G06, GL-3uniform,KMV} among others.
Our work in this paper focuses on so-called {\it Berge cycles}. The method we develop here also
works for some other common notions of cycles, such as so-called {\it linear cycles} (or  sometimes known as {\it loose cycles}).
See Section 6 for discussion in that direction.

A hypergraph $\CH=(\CV,\CE)$ consists of a set $\CV$ of {\it vertices} and a collection $\CE$ of subsets of $\CV$.
We call a member of $\CE$ a {\it hyperedge} or simple an {\it edge} of $\CH$. A hypergraph $\CH$ is
{\it $r$-uniform} if all of its edges are $r$-subsets of $\CV(H)$. We also simply call an $r$-uniform hypergraph
an {\it $r$-graph} for brevity.
%Recently considerable attention has been given to the extremal problems on paths and cycles
%in hypergraphs.
A {\it Berge path of length $\ell$},  is a hypergraph $\CP$ consisting of $\ell$ distinct edges $e_1,\dots, e_\ell$
such that there exist $\ell+1$ distinct vertices $v_1,\dots, v_{\ell+1}$ satisfying that
$v_i, v_{i+1}\in e_i$ for $i=1,\ldots, \ell$. We call the $2$-uniform path $v_1v_2\dots v_{\ell+1}$ a {\it spine} of the Berge path $\CP$.
A {\it Berge cycle of length $\ell$} is a hypergraph $\CC$ consisting of
$\ell$ distinct hyperedges $e_1,...,e_\ell$ such that there exist $\ell$ distinct vertices $v_1,\dots, v_\ell$ satisfying that
$v_i,v_{i+1}\in e_i$ for each $i=1,...,\ell-1$ and $v_1,v_\ell\in e_\ell$. We call the $2$-uniform cycle $v_1v_2\dots v_\ell v_1$
a {\it spine} of the Berge cycle $\CC$.

Let us now recall Verstra\"ete's theorem on cycles of consecutive even lengths in graphs.

\begin{theorem} {\rm (Verstra\"ete \cite{V00})} \label{Jacques-even-cycles}
Let $k\geq 2$ be a natural number and $G$ a bipartite graph of average degree at least $4k$ and radius $h$.
Then $G$ contains cycles of $k$ consecutive even length. Moreover, the shortest of these
cycles has length at most $2h$.
\end{theorem}
Let us note that Verstra\"ete's result, together with a simple induction, immediately yields
$ex(n,C_{2k})\leq 8kn^{1+1/k}$, thereby giving a short proof of the theorem of Bondy and Simonovits \cite{BS74}
along with an improved coefficient.
In an attempt to generalize Theorem \ref{Jacques-even-cycles}, Verstra\"ete \cite{V16} made the following conjecture for Berge cycles in $r$-graphs.

\begin{conjecture}{\rm (Verstra\"ete \cite{V16})} \label{Jacques-conjecture}
\label{conj:Berge}
Let $r\ge 3$. If $\CH$ is an $r$-graph which does not contain Berge cycles
of $k$ consecutive lengths, then H has average degree $O(k^{r-1})$ as $k\to \infty$.
\end{conjecture}

\noindent The complete $r$-graph on $k$ vertices shows that this conjecture, if true, is best
possible up to some constant factor (depending only on $r$).

Our main result is to confirm Conjecture \ref{Jacques-conjecture} by showing

\begin{theorem}\label{Thm:main}
Let $r\ge 3$. Any $r$-graph $\CH$ with average degree $\Omega(k^{r-1})$ contains Berge cycles of $k$ consecutive lengths.
\end{theorem}

Moreover, we will establish a stronger statement, in which, as in Theorem \ref{Jacques-even-cycles}, we have
additional control on the length of the shortest cycle in the collection, in terms of a parameter of a hypergraph
that is related to the radius of a graph. This stronger version then immediately yields bounds on the Tur\'an numbers
of Berge cycles (for both even and odd cycles simultaneously). See Section 5 for detailed discussions.

A hypergraph $\CH$ is {\it linear}, if $|e\cap f|\le 1$ for any distinct hyperedges $e,f\in \CE(\CH)$.
For linear hypergraphs, we have the following even stronger result.

\begin{theorem}\label{Thm:Li3Berge}
Let $r\geq 3$.
Any linear $r$-graph $\CH$ with average degree at least $7r(k+1)$ contains Berge cycles of $k$ consecutive lengths.
\end{theorem}

\noindent This is also tight up to the constant factor, by considering a Steiner triple system on $k$ vertices.
As with Theorem \ref{Thm:main}, we will also establish a stronger version of Theorem \ref{Thm:Li3Berge}
with control on the length of the shortest cycle in the collection. 

As we should see,  Theorem \ref{Thm:Li3Berge} plays a central role in our results.
In fact, we will prove Theorem \ref{Thm:main} in Section 4 by reducing it to the $r=3$ case of Theorem \ref{Thm:Li3Berge}
via some reduction lemmas. These reduction lemmas given in Section 4 may be of independent interest in the study of
Tur\'an numbers of other Berge hypergraphs.

Let us now say a few words about the method we use in this paper, which we feel should find many future applications.
In many Tur\'an type results on cycles in graphs, the following lemma plays an important role. The lemma was implicit
in Bondy and Simonovits \cite{BS74} and was explicit in Verstra\"ete \cite{V00}.
Let $G$ be a graph and $A, B$ be disjoint subsets of $V(G)$.
An {\it$(A,B)$-path} is a path that has one endpoint in $A$ and the other in $B$.

\begin{lemma} {\bf (The $(A,B)$-path lemma)} {\rm (\cite{BS74}, see also \cite{V00})} \label{BS}
Let $H$ be a graph comprising a cycle with a chord. Let $(A,B)$ be a nontrivial partition of $V(H)$. Then $H$ contains
$(A,B)$-paths of every length less than $|V(H)|$, unless $H$ is bipartite with bipartition $(A,B)$.
\end{lemma}

Despite its highly successful applications,
one restriction of this powerful lemma is, however, that it does not easily generalize to hypergraphs, especially for cycle structures that are more restrictive than
Berge cycles. Faudree and Simonovits \cite{FS} introduced an alternative method, called the blowup method. They used the method
to study the Tur\'an number of a theta graph (which is the union of a number of internally disjoint paths of the same length between two fixed vertices).
The method of Faudree and Simonovits allows for applications to hypergraphs more easily (see \cite{CGJ} for such an application). However, the constant
involved in the Faudree-Simonovits method is usually very large. Therefore, it is desirable to look for new methods to add to the existing ones.

The method we use in this paper builds on the ones used in
\cite{GL-3uniform,GL,FO}, but  contains some novel ingredients and
overcomes some drawbacks of the existing methods.
The key ideas of our method are contained in the proofs of Lemma \ref{twice-radius-A} and Lemma \ref{twice-radius-BC},
both of which rely on a simple lemma (Lemma \ref{special-paths} and its hypergraph extension), used
in place of the $(A,B)$-path lemma of Bondy and Simonovits.
We expect our method to find further applications in both the graph and hypergraph settings.

The rest of the paper is organized as follows.
In Section 2, we introduce notation, some lemmas,
and a generalization of the breadth-first search tree for $3$-graphs.
In Section 3, we prove Theorem \ref{Thm:Li3Berge}.
In Section 4, we prove Theorem \ref{Thm:main}.
In Section 5, we discuss the Tur\'an numbers of Berge cycles in $r$-graphs and the Zarankiewicz numbers of even cycles.
In Section 6, we discuss some future directions.
Throughout this paper we make some but not significant efforts to optimize the constant factors used in the proofs.

\section{Preliminaries}
\subsection{Notation and terminologies}

Unless otherwise specified,  all hypergraphs $\CH$ discussed in this paper are {\it simple},
that is, all edges of $\CH$ are distinct subsets.
The {\it $k$-shadow} $\partial_k(\CH)$ of $\CH$ denotes the $k$-graph on the same vertex set of $\CH$,
whose edge set consists of all subsets of $k$ vertices contained in members of $\CE(\CH)$.
If $k=2$, we then just write it as $\partial \CH$ and call it {\it shadow} of $\CH$.
We say a subhypergraph $\CG$ of $\partial_k(\CH)$ is {\it extendable} in $\CH$,
if there exists an injection $\psi: \CE(\CG)\to \CE(\CH)$ such that $e\subseteq \psi(e)$ for each $e\in \CE(\CG)$.
Such an injection $\psi$ is called an {\it extension} of $\CG$. It is immediate from our definition that
$$\CH \text{ contains a Berge cycle of length } \ell \Longleftrightarrow \partial \CH \text{ contains an extendable cycle } C_\ell.$$

Let $\CH$ be a hypergraph. By $|\CH|$, we mean the total number of hyperedges in $\CH$.
For a vertex $v\in \CV(\CH)$, the {\it degree} $d_{\CH}(v)$ denotes the number of hyperedges in $\CH$ containing the vertex $v$.
We use $\delta(\CH)$ and $d(\CH)$ to denote the {\it minimum degree} and the {\it average degree} of $\CH$, respectively.
For a subset $S\subseteq \CV(\CH)$ of size at least two, the {\it co-degree}
$d_{\CH}(S)$ denotes the number of hyperedges of $\CH$ containing $S$.
If $S=\{u,v\}$, then we will just write it as $d_{\CH}(u,v)$.
For $k\ge 2$, the {\it minimum $k$-degree} $\delta_k(\CH)$ of $\CH$ is the minimum of non-zero
co-degrees $d_{\CH}(S)$ over all subsets $S\subseteq \CV(\CH)$ of size $k$.
A subset $S\subseteq \CV(\CH)$ is a {\it vertex cover} of $\CH$
if each edge of $\CH$ contains a vertex in $S$.

A {\it linear path} of length $\ell$ is a hypergraph with edges $e_1,e_2,..., e_\ell$ such that $|e_i\cap e_{i+1}|=1$ for every $1\le i\le \ell-1$,
and $e_i\cap e_j=\emptyset$ for other pairs $\{i,j\}$, where $i\neq j$.
A vertex in the first or last edge of a linear path $\CP$ that has degree $1$ in $\CP$ is called an {\it endpoint} of $\CP$.
%Certainly a linear path is a Berge path.
A {\it linear cycle} of length $\ell$ is a hypergraph with edges $e_1,e_2,..., e_\ell$ such that $|e_i\cap e_{i+1}|=1$ for every $1\le i\le \ell-1$,
$|e_1\cap e_\ell|=1$, and $e_i\cap e_j=\emptyset$ for other pairs $\{i,j\}$, where $i\neq j$.
Certainly a linear path/cycle is also a Berge path/cycle.
\footnote{The {\it length} of a path or a cycle (of any kind) always denotes the number of edges it contains.}

%%%%%%%%%%%%%%%

\subsection{Some lemmas on paths and cycles in hypergraphs}
In this subsection, we collect and establish several lemmas to be used later,
including the pivotal Lemma \ref{special-paths} and its hypergraph extension as mentioned in the introduction.

An $r$-graph $\CH$ is {\it $r$-partite}, if there exists an $r$-partition $(A_1, A_2,...,A_r)$ of $\CV(\CH)$ such that
for any edge $e\in \CE(\CH)$, $|e\cap A_i|=1$ for each $i\in [r]$.
The following lemma from \cite{EK} is well-known.

\begin{lemma}\label{r-partite} {\rm (\cite{EK})}
Every $r$-graph $\CH$ has an $r$-partite subhypergraph $\CH'$ satisfying $|\CH'|\ge \frac{r!}{r^r}\cdot |\CH|$.
\end{lemma}

\begin{lemma}\label{prop:ave->min}
Every $r$-graph $\CH$ has a subhypergraph $\CH'$ satisfying that $d(\CH')\geq d(\CH)$ and $\delta(\CH')\ge \frac{1}{r}\cdot d(\CH)$.
\end{lemma}

\begin{proof} Let $d=d(\CH)$. Then $|\CH|= |\CV(\CH)|\cdot d/r$.
Let $\CH'$ be a smallest subhypergraph of $\CH$ satisfying $|\CH'|\geq |\CV(\CH')|\cdot d/r$;
$\CH'$ exists since $\CH$ satisfies the inequality. If $\CH'$ contains a vertex $x$ of degree less than $d/r$,
then deleting $x$ from $\CH'$ yields a smaller subhypergraph $\CH''$ with $|\CH''|\geq |\CV(\CH'')|\cdot d/r$,
a contradiction. Hence $\delta(\CH')\geq d/r$. Also, since $|\CH'|\geq |\CV(\CH')|\cdot d/r$, we have
$d(\CH')\geq d$.
\end{proof}

We need the following classic result of Erd\H{o}s and Gallai \cite{EG}.

\begin{theorem} {\rm \cite{EG}} \label{EG-theorem}
Every $n$-vertex graph with more than $m(n-1)/2$ edges has a cycle of length at least $m+1$.
Every $n$-vertex graph with more than $mn/2$ edges has a path of length at least $m+1$.
\end{theorem}

\begin{lemma}\label{prop:lin3par}
Let $\CH$ be a linear $3$-partite 3-graph with a $3$-partition $(V_1,V_2,V_3)$ such that $d(\CH)\ge \frac{3p}{2}$.
There exists a vertex $u\in V_1$ such that for each $\ell\in [p]$, there exists an extendable path of length $\ell$ in $\partial \CH$ from $u$ to a vertex in $V_2$.
\end{lemma}
\begin{proof}
Let $n=|\CV(\CH)|$. Let $G=\{(e\cap (V_1\cup V_3): e\in \CH\}$. Since $\CH$ is linear, $|G|=|\CH|\geq \frac{pn}{2}>\frac{pn(G)}{2}$.
Let $t=\ce{\frac{p}{2}}$. Then $2t\leq p+1$.
 Theorem \ref{EG-theorem}, $G$ contains a path $P$ of length $2t$. Let $P=u_1v_1u_2v_2...u_tv_t$.
By symmetry, we may assume that $u_1\in V_1$. Then $\forall i\in [s], u_i\in V_1, v_i\in V_3$.
 Since $\CH$ is $3$-partite and linear, different edges of $P$ extend to different edges of $\CH$.

Let $s\in [t]$, we construct extendable paths of length $2s-1$ and $2s$, respectively, in $\partial\CH$ from $u_1$ to a vertex in $V_2$.
Let $u_s v_s w$ be the unique edge of $\CH$ containing $u_s v_s$, where $w\in V_2$.
Let $P_{2s-1}=u_1v_1...u_s w$. Then $P_{2s-1}$ has length $2s-1$ and it is easy to see that it is extendable.
Next, let $v_s u_{s+1} w'$ be the unique edge of $\CH$ containing $v_su_{s+1}$, where $w'\in V_2$.
Let $P_{2s}=u_1v_1...u_s v_s w'$. Then $P_{2s}$ has length $2s$ and is extendable, completing the proof.
\end{proof}

We say that a hypergraph $\CH$ is {\it connected} if $\partial\CH$ is connected.

\begin{lemma} \label{linear-uv-path}
Let $\CH$ be a connected linear hypergraph and $X,Y$ two disjoint sets of vertices in $\CH$.
Then there exists a linear $(X,Y)$-path in $\CH$.
\end{lemma}
\begin{proof}
Since $\CH$ is connected, $\partial \CH$ is a connected $2$-graph. Let $P$ be a shortest $(X,Y)$-path in
$\partial \CH$. Suppose $P=u_1u_2\dots u_m$ where $u_1\in X, u_m\in Y$. For each $i\in [m-1]$, let
$e_i$ denote the unique edge of $\CH$ containing $u_iu_{i+1}$. If $e_i\cap V(P)\neq \{u_i, u_{i+1}\}$ for some $i\in [m-1]$,
then we can find a shorter $(X,Y)$-path in $\partial\CH$ than $P$, a contradiction.
Likewise if $e_i\cap e_j$ contains a vertex outside $V(P)$ then since $\CH$ is linear, $|i-j|\geq 2$,
in which case we can find a shorter $(X,Y)$-path in $\partial\CH$ than $P$, a contradiction.
Therefore, $e_1,e_2,\dots, e_m$ form a linear $(X,Y)$-path in $\CH$.
\end{proof}

The next lemma can be viewed as a replacement for the $(A,B)$-path lemma (i.e., Lemma \ref{BS}).
This lemma and its hypergraph extension will be crucial in our proofs.
\footnote{Note that this lemma yields a better constant than its hypergraph extension.}

\begin{lemma} \label{special-paths}
Let $G$ be a connected graph whose edges are colored with $1$ and $2$ such that there is at least an edge of each color.
For $i\in [2]$, let $G_i$ denote the subgraph of $G$ consisting of edges of color $i$.
If $d(G_1)\geq p+1$, then there exists a path of length at least $p$ in $G$ such that the first edge has color $2$ and the others all have color $1$.
\end{lemma}
\begin{proof}
Let $n=|V(G)|$. By our assumption, $e(G_1)\geq (p+1)n/2$ and that there is an edge $xy$ of color $2$.
Let $G'_1=G_1-x$. Then $e(G'_1)>(p-1)n/2$. By Theorem \ref{EG-theorem}, $G'_1$ contains a cycle $C$ of
length at least $p$. Since $G$ is connected, there exist paths in $G$ from $\{x,y\}$ to $V(C)$.
Among them, let $P$ be a shortest one. Let $z$ be the unique vertex in $V(C)\cap V(P)$.
Let $z'$ be a neighbor of $z$ on $C$.
If $P$ has none of edges of color 2,
then $(P\cup C\cup \{xy\})\setminus \{zz'\}$ satisfies the requirement.
So we may assume that $P$ has an edge of color $2$. In this case, let $P'$ be a shortest
subpath on $P$ containing $z$ and a color $2$ edge. Then $(P'\cup C)\setminus \{zz'\}$
satisfies the requirement.
\end{proof}

We next give a hypergraph extension of Lemma \ref{special-paths}. Even though the objects we study are
Berge paths and Berge cycles, it appears that developing variants of Lemma \ref{special-paths} for linear paths
would facilitate the arguments better and may be handy for future study on linear cycles.
%Such a lemma will also be handy in any future study of
%degree conditions for the existence of linear cycles of consecutive lengths in uniform hypergraphs.

\begin{lemma} \label{special-paths-hyper} {\bf (The special-path lemma)}
Let $r,p\geq 2$. Let $\CH$ be a connected linear $r$-graph whose edges are colored with $1$ and $2$
such that there is at least one edge of each color. For $i\in [2]$, let $\CH_i$ denote the
subhypergraph of $\CH$ consisting of edges of color $i$. Suppose $d(\CH_1)\geq r(r-1)(p-1)+2r$.
Then there exists a linear path of length at least $p$
such that the first edge has color $2$ and the other edges all have color $1$.
\end{lemma}
\begin{proof} Let $n=|\CV(\CH)|$.
Let $h$ be an edge of color $2$. Let $u$ be a vertex in $h$. Let $\CH'$ be obtained from $\CH_1$ by deleting the $r-1$ vertices
in $h$ besides $u$. Since $\CH_1$ is linear, we lose at most $\frac{n-1}{r-1}$ edges by deleting any vertex. Since $|\CH_1|\geq (r-1)(p-1)n+2n$,
certainly we have $|\CH'|>(r-1)(p-1)n+n$.
By Lemma \ref{prop:ave->min}, $\CH'$ contains a subhypergraph  $\CH''$ with minimum degree at least $(r-1)(p-1)+2$.
Call a linear path $\CP$ in $\CH$ {\it good} if its first edge has color $2$ and the other edges have
color $1$ and if the last edge contains an endpoint lying in $\CV(\CH'')$.

First we show that there exists at least one good path in $\CH$. Let $\CV(h)$ denote the set of vertices in $h$.
Since $\CH$ is connected, by Lemma \ref{linear-uv-path}, there exist linear paths in $\CH$ from $\CV(h)$ to $\CV(\CH'')$. Among these paths
let $\CP$ be a shortest one.  Let $e$ denote the unique edge of $\CP$ intersecting $\CV(h)$ and view it
as the first edge of $\CP$ and let $f$ denote the unique edge of $\CP$ intersecting $\CV(\CH'')$. It is possible that $e=f$.
Let $v$ be a vertex in $f\cap \CV(\CH'')$. By our choice of $\CP$, $v$ has degree $1$ in $\CP$.
First suppose that all edges on $\CP$ have color $1$. By our choice of $\CP$, $\CQ=\CP\cup h$ is  a linear path
that starts with a color $2$ edge (namely $h$) but have all color $1$ edges otherwise and one of the endpoints of the last edge
is in $\CV(\CH'')$.  Hence $\CQ$ is a good path.
Next, suppose $\CP$ contains at least one color $2$ edge. In this case, let $\CQ$ be the shortest linear-path contained in $\CP$
that contains a color $2$ edge and the edge $f$. Then $\CQ$ is a good path.

Now, among all good paths in $\CH$, let $\CQ^*$ be a longest one. Let $e,f$ denote the first and last edges, respectively.
By our assumption, $e$ is the only edge on $\CQ^*$ with color $2$ and $f$ contains an endpoint $z$ in $\CV(\CH'')$.
Suppose first $\CQ^*$ has length at most $p-1$. Since $\delta(\CH'')\geq (r-1)(p-1)+2$,
there are at least $(r-1)(p-1)+1$ edges of $\CH''$ besides $f$ that contain $z$.
Since there are at most $(r-1)(p-1)$ vertices in $\CV(\CQ^*)\setminus \{z\}$ and $\CH$ is linear, one of these edges $e'$ is disjoint
from $\CV(\CQ^*)\setminus \{z\}$. But then $\CQ^*\cup \{e'\}$ is a good path in $\CH$ that is longer than $\CQ^*$, a contradiction.
Hence $\CQ^*$ must have length at least $p$ and it  is a linear path that satisfies the claim.
\end{proof}

%%%%%%%%%%%%%%%%%%%%%

\subsection{Maximal extendable skeletons}

We now define a generalization of the breadth-first search tree for $3$-graphs,
which was first introduced by Gy\H{o}ri and Lemons in \cite{GL-3uniform}.
Let $\CH$ be a $3$-graph and $x$ a vertex of $\CH$. A {\it maximal extendable skeleton} of $\CH$
rooted at $x$ is an extendable subgraph $T\subseteq \partial \CH$,
obtained by running the following algorithm till its termination.

\begin{algorithm} \label{BFS} {\rm (Modified Breadth First Search)\\
{\it Input: }A $3$-graph $\CH$ and a vertex $x\in V(\CH)$.\\
{\it Output:} A tree $T$ rooted at $x$ and an extension  $\psi: E(T)\to \CE(\CH)$.\\
\\
{\it Initiation:} Set $T=\emptyset$, $Q=\{x\}$. Label all edges of $\CH$ as unprocessed. \\
{\it Iteration:} Maintain $Q$ as a queue.
Let $u$ be the first vertex in $Q$. Consider any unprocessed edge $e=uvw$ containing $u$.
If at least one of $v,w$ is not in $V(T)$, then pick one, say $v$.
Add edge $uv$ to $T$ and add $v$ to the end of $Q$. Let $\psi(uv)=e$ and  mark $e$ as processed and delete $e$ from $\CH$.
If both of $v,w$ are already in $V(T)$, then skip edge $e$ and mark it as processed. When all edges
containing $u$ have been processed, we delete $u$ from $Q$.\\
{\it Termination:} Terminate when $Q=\emptyset$. \qed
}
\end{algorithm}

Clearly, the graph $T$ we obtain from Algorithm \ref{BFS} is a tree rooted at $x$, in which the
distance from a vertex $y$ to the root $x$ is non-decreasing in the order in which $y$ is being added to $T$
(or equivalently to the queue $Q$).
We point out that such tree $T$ may not be spanning in $\partial \CH$.

Given a $3$-graph $\CH$ and a maximal extendable skeleton $T$ in $\CH$ rooted at some vertex $x$
together with an extension $\psi$ of it, let
\begin{equation} \label{ET}
\CE_{T,\psi}=\{\psi(e): e\in E(T)\}
\end{equation}
and
\begin{equation}\label{ET-prime}
\CE'_{T,\psi}=\{e\in \CE(\CH)\setminus \CE_{T,\psi}:  e\cap V(T)\neq \emptyset\}.
\end{equation}

Note  that if $\CH$ is a linear $3$-graph, then each edge in $\partial \CH$ lies in a unique edge of $\CH$.
So an extension $\psi$ of the  a maximal extendable skeleton $T$ in a linear $3$-graph $\CH$ is uniquely determined by $T$.
In this case, we will simply write $\CE_T$ and $\CE'_T$ for $\CE_{T,\psi}$ and $\CE'_{T,\psi}$, respectively.

Given two vertices $u,v\in V(T)$, let $d_T(u,v)$ denote the length of the unique $(u,v)$-path in $T$.
For each integer $i\geq 0$, let
\begin{equation} \label{levels}
L_i(T)=\{y\in V(T): d_T(x,y)=i\}.
\end{equation}
When the context is clear, we will simply write $L_i$ for $L_i(T)$. We call $L_i$ {\it level $i$} of $T$.

\begin{proposition} \label{level-containment}
Let $\CH$ be a $3$-graph and $x\in V(\CH)$. Let $T$ be a maximal extendable skeleton in $\CH$
rooted at $x$ together with an extension $\psi$.
For every $e\in \CE'_{T,\psi}$, there exists some $i\geq 0$ such that $e\subseteq L_i(T)\cup L_{i+1}(T)$.
\end{proposition}
\begin{proof}
We already observed that for any $y,y'\in V(T)$, if $y$ is added to $T$ before $y'$ then
$d_T(x,y)\leq d_T(x,y')$.
Among all vertices in $e\cap V(T)$, let $u$ be the vertex that is added to $T$ the earliest.
Let $i=d_T(x,u)$. Then $u\in L_i(T)$. Suppose $e=uvw$. Consider the time the algorithm
processed $e$. If one of $v$ or $w$, say $v$, was not in $T$ at that time,
then the algorithm would have added $uv$ to $T$, deleted $e$, and let $\psi(uv)=e$,
contradicting $e\in \CE'_{T,\psi}$. Hence, $v,w$ were both in $V(T)$ at that time.
By our earlier discussion, $d_T(x,v),d_T(x,w)\geq i$. Since $v$ was already in $T$, it
must have been added to $T$ when we processed an edge $e'$ containing a vertex $u'$
where $u'$ was either $u$ itself or an earlier vertex in the queue $Q$. Hence $d_T(x,v)=d_T(x,u')+1\leq i+1$.
Similarly, $d_T(x,w)\leq i+1$.
\end{proof}

%%%%%%%%%%%%%%%%%%%%%%%%%%%%%%%%%%%%%%%%%%%%
\section{Consecutive cycles in linear $r$-graphs}
We devote this section to the proof of Theorem \ref{Thm:Li3Berge}.
The next two lemmas are crucial and their proofs contain the main ideas of our method used in this paper.

\begin{lemma} \label{twice-radius-A}
Let $\CH$ be a linear $3$-graph and $T$ a maximal extendable skeleton in $\CH$ rooted
at some vertex $r$. For each $i\geq 1$, let $L_i=\{v: d_T(r,v)=i\}$ and $\CA_i=\{e\in \CH: |e\cap L_i|=2, |e\cap L_{i-1}|=1\}$.
If there exists an $i\geq 1$ such that $|\CA_i|\geq (k+2)|L_i|$, then $\CH$ contains Berge cycles of
all lengths in $[a,a+k-1]$ for some $a\leq 2i$.
\end{lemma}
\begin{proof}
Fix an $i\geq 1$ that satisfies $|\CA_i|\geq (k+2)|L_i|$. Since $\CH$ is linear, the condition implies that $i\geq 2$.
Let $G_i=\{e\cap L_i: e\in \CA_i\}$. Since $\CH$ is linear, there is bijection between $\CA_i$ and $G_i$.
So, $G_i$ is a $2$-graph on $L_i$ with $|G_i|=|\CA_i|\geq (k+2)|L_i|$.
Hence, $d(G_i)\geq 2k+4$. Let $G$ be a connected component of $G_i$ with $d(G)\geq 2k+4$.
By our assumption, each $uv\in E(G)$ has a unique co-neighbor $w$ in $\CA_i$ and $w\in L_{i-1}$.
Let $S$ be the set of co-neighbors  of $uv$ over all edges $uv\in E(G)$.
Note that the condition $d(G)\geq 2k+4$ implies that $|S|\geq 2$. Indeed, the link of any single vertex can only be a matching.

Let $r^*$ be a closest common ancestor in $T$ of vertices in $S$ and let $T^*$ be unique subtree of $T$
rooted at $r^*$ whose set of leaves is $S$. Suppose $r^*\in L_{i'}$.
Let $s_1,\dots, s_m$ denote the children of $r^*$ in $T^*$.
Color all the descendants of $s_1$ in $T^*$ that lie in $S$ color $1$ and the other vertices in $S$ color $2$.
Call this coloring $c$. By our choice of $r^*$ and the fact that $|S|\geq 2$, we have $m\geq 2$. So both color $1$
and color $2$ appear in $S$. Now, define an edge-coloring $\phi$ of $G$ as follows.
For each $uv\in E(G)$, let  $w$ be the unique vertex in $S$ such that $uvw\in \CA_i$. Let $\phi(uv)=c(w)$.
This defines a $2$-coloring of the edges of $G$ in which there is at least one edge of each color.
For $i\in [2]$, let $G_i$ denote the subgraph of $G$ consisting of edges of color $i$.
We may assume that $e(G_1)\geq e(G_2)$ since the arguments for  the $e(G_2)\geq e(G_1)$ case are identical.
By our assumption, $d(G_1)\geq k+2$. Since $G$ is connected, by Lemma \ref{special-paths},
there exists a path $P$ of length $k+1$ in $G$ such that first edge has color $2$ and all the others have color $1$.
Suppose $P=uv_1v_2\dots v_{k+1}$, where $uv_1$ is the only edge of color $2$. Let $x$ be the unique vertex in $S$
such that $uv_1x\in \CA_i$. Let $X$ denote the unique path in $T^*$ from $r^*$ to $x$.
For each $\ell\in [k]$, let $y_\ell$ be the unique vertex in $S$ such that $v_\ell v_{\ell+1} y_\ell\in \CA_i$.
Let $Y_\ell$ denote the unique path in $T^*$ from $r^*$ to $y_\ell$.

Now, $P_\ell=xv_1\dots v_\ell y_\ell$ is an extendable path of length $\ell+1$ whose edges extend to different hyperedges of
$\CA_i\subseteq \CE'_T$. Clearly $X\cup Y_\ell$ is an extendable path of length $2(i-1-i')$ whose edges extend
to different hyperedges in $\CE_T$. Further, $V(X\cup Y_\ell)\cap V(P_\ell)=\{x,y_\ell\}$. Hence,
$P_\ell\cup X\cup Y_\ell$ is an extendable cycle in $\partial(\CH)$ of length $2(i-1-i')+\ell+1$, where $\ell\in [k]$.
So $\CH$ contains a Berge cycle of each length in $[a, a+k-1]$, where $a=2(i-i')\leq 2i$.
\end{proof}

\begin{lemma} \label{twice-radius-BC}
Let $\CH$ be a linear $3$-graph and $T$ a maximal extendable skeleton in $\CH$ rooted at some vertex $r$.
For each $i\geq 1$, let $L_i=\{v: d_T(r,v)=i\}$, $\CB_i=\{e\in \CH: e\subseteq L_i\}$, and
$\CC_i=\{e\in \CH: |e\cap L_i|=2, |e\cap L_{i+1}|=1\}$. For each $i\geq 1$,
if $\CH$ contains no Berge cycles of all lengths in $[a,a+k-1]$ for any $a\leq 2i+2$,
then
 $|\CB_i\cup \CC_i|\leq (\frac{7k}{2}+1)|L_i|+(\frac{5k}{2}+2)|L_{i+1}|$.
\end{lemma}
\begin{proof}
Fix an $i$ for which $\CH$ does not contains Berge cycles of all lengths in $[a,a+k-1]$ for any $a\leq 2i+2$.
Let $\CG$ be a connected nontrivial component of $\CB_i\cup \CC_i$. Let $n_i=|\CV(\CG)\cap L_i|$
and $n_{i+1}=|\CV(\CG)\cap L_{i+1}|$. Our goal is to show $|\CG|\le (\frac{7k}{2}+1)n_i+(\frac{5k}{2}+2)n_{i+1}$.

Let $r^*$ denote the closest common ancestor of $\CV(\CG)$ in $T$. Suppose $r^*\in L_{i'}$. Let $T^*$
be the minimal subtree of $T$ rooted at $r^*$ which contains $\CV(\CG)$. Let $s_1,\dots, s_m$ denote
the children of $r^*$ in $T^*$. By our choice of $r^*$, $m\geq 2$. Now, color all descendants of
$s_1$ in $V(T^*)$ with color $1$ and the descendants of other $s_j$'s in $V(T^*)$ with color 2.
Call this coloring $c$. By the minimality of $T$, we see that $\CG$ has vertices of both colors.

\medskip

We now summarize one ingredient that we will repeatedly use into the following Claim.

\medskip

{\bf Claim 1.} Suppose $\partial \CG$ contains an extendable path $P$ of length $\ell$ with endpoints $x,y$
such that $c(x)\neq c(y)$ and $V(X\cup Y)\cap V(P)=\{x,y\}$, where $X,Y$ denote the unique paths in $T^*$ from
$r^*$ to $x,y$, respectively. Then $X\cup Y\cup P$ is an extendable cycle of length $(i(x)+i(y)-2i')+\ell$,
where $i(x), i(y)$ denote the levels $x,y$ are in, respectively.

\medskip

{\it Proof.} Since $c(x)\neq c(y)$, $V(X)\cap V(Y)=\{r^*\}$. Thus $C=X\cup Y\cup P$ is a cycle in $\partial \CH$ of
length $|X|+|Y|+\ell$. By our assumption different edges on $X\cup Y$ extend to different hyperedges of $\CE_T$,
while different edges on $P$ extend to different hyperedges of $\CE'_T$. Hence $C$ is an extendable cycle in
$\partial \CH$. Lastly, note that $|X|=i(x)-i'$ and $|Y|=i(y)-i'$.\qed

\medskip

For each edge $e\in \CG$, we call it {\it monochromatic}
if all three vertices in $e$ have the same color in $c$; otherwise we call it {\it non-monochromatic}.
Let $\CM$ denote the subgraph of $\CG$ consisting of monochromatic edges and $\CN$ the subgraph of
$\CG$ consisting of non-monochromatic edges. So $\CG=\CM\cup \CN$, and by our choice of $r^*$, $\CN\neq \emptyset$.

\medskip

{\bf Claim 2.} $|\CM|\leq (2k+2)(n_i+n_{i+1})$.

\medskip

{\it Proof.} Suppose that $|\CM|>(2k+2)(n_i+n_{i+1})$. We will derive a contradiction.
By Lemma \ref{special-paths-hyper}, there exists a linear path
$\CP$ in $\CG$ of length at least $k+1$ such that the first edge is in $\CN$ and all the other edges are in $\CM$.
Let the edges of $\CP$ be $e_1,e_2,\dots, e_{k+1}$ in order where $e_1$ is the edge in $\CN$. It is easy to see
that all the vertices in $\CV(\CP\setminus e_1)$ must have the same color under $c$. By renaming the colors in $c$
if necessary, we may assume that they have color $1$.
%\footnote{Observant readers may find that color-1 vertices and color-2 vertices are not symmetric from their definitions.
%However, for the purpose of our presentation we will view them symmetric, as all one needs in the proof
%is that any two vertices of different colors have two internally disjoint subpaths of $T^*$ to the root $r^*$.}
Also, among the two endpoints $x,x'$ of $\CP$ contained in $e_1$, either $c(x)\neq 1$ or $c(x')\neq 1$.
Without loss of generality, suppose that $c(x)=2$.  For each $\ell\in [k+1]$, let $\CP_\ell$ denote the linear path consisting of  $e_1,\dots, e_\ell$.
By our definition of $\CG$, $x\in L_i\cup L_{i+1}$.

First suppose that $x\in L_{i+1}$.  Let $y_1$ be any vertex in $e_1$ that has color $1$ and lies in $L_i$ (note that such a vertex exists).
Let $\ell\in \{2,\dots, k\}$. By our discussion, all three vertices in $e_\ell$ have color $1$, and among the two endpoints in $e_{\ell}$ of $\CP_\ell$,
there exists one that lies in $L_i$, denote that vertex by $y_\ell$. Now $\CP_\ell$ is a linear $(x,y_\ell)$-path of length $\ell$ in $\CG$
whose vertices are contained in $L_i\cup L_{i+1}$. Let $P_\ell$ denote a spine of it with $x,y_\ell$ as endpoints.
Note that all vertices on $P_\ell$ except $x$ have color $1$ in $c$ and $c(x)=2$.
Let $X, Y_\ell$ denote the unique paths in $T^*$ from $r^*$ to $x$ and $y_\ell$,
respectively. By the definition of $Y_\ell$, $Y_\ell$ intersects $P_\ell$ only at $y_\ell$.
Also, $X$ contains a vertex each from $L_i$ and $L_{i+1}$,
both of which have color $2$ in $c$, but all vertices on $P_\ell$ other than $x$ have color $1$. Hence $X$ intersects $P_\ell$ only at $x$.
By Claim 1, $X\cup Y_\ell\cup P_\ell$ is an extendable cycle of length $(i-i')+(i+1-i')+\ell$ in $\partial \CH$.
Since this holds for each $\ell\in [k]$,  we get Berge cycles of all lengths in $[2(i-i')+2, 2(i-i')+k+1]$.
This contradicts our assumption that $\CH$ has no Berge cycles of $k$ consecutive lengths the shortest of which has
length at most $2i+2$.

Next, suppose $x\in L_i$. Let $\ell \in \{2,\dots, k+1\}$. By definition all the vertices in $e_\ell$ have color $1$ under $c$.
Among the two endpoints in $e_\ell$ of $\CP_\ell$,
one of them lies in $L_i$. Let $y_\ell$ denote such a vertex. By a similar argument as above, we can obtain a Berge cycle
of length $(i-i')+(i-i')+\ell$ for each $\ell \in \{2,\dots, k+1\}$. Hence $\CH$ has Berge cycles of all lengths in  $[2(i-i')+2, 2(i-i')+k+1]$.
Again, this contradicts our assumption that $\CH$ has no Berge cycles of $k$ consecutive lengths the shortest of which has
length at most $2i+2$. This proves Claim 2. \qed

\medskip

{\bf Claim 3.} $|\CN|\leq (\frac{3k}{2}-1)n_i+\frac{k}{2} n_{i+1}$.

\medskip

{\it Proof.} We further decompose $\CN$ as follows. Let $S_1,S_2$ denote the set of descendants of $r^*$ in $L_i$ with color $1$ and $2$, respectively.
Let
\begin{eqnarray*}
\CN_1&=&\{e\in \CN: |e\cap S_1|=2\}\\
\CN_2&=&\{e\in \CN: |e\cap S_2|=2\}\\
\CN_3&=&\{e\in \CN\cap \CC_i: |e\cap S_1|=|e\cap S_2|=1\}
\end{eqnarray*}
By definition, any $e\in \CN\cap \CB_i$ intersects both $S_1$ and $S_2$.
Thus we see that $\CN=\CN_1\cup \CN_2\cup \CN_3$.
Suppose first that $|\CN_1|> \frac{k-1}{2} n_i$. Let $G_1$ be formed by taking the pair of color $1$ vertices
from each hyperedge in $\CN_1$. Then $|G_1|=|\CN_1|>\frac{k-1}{2}n_i>\frac{k-1}{2} n(G_1)$. By Theorem \ref{EG-theorem}, $G_1$
contains a path $P$ of length at least $k$. Let $P=y_0y_1\dots y_k$. Let $xy_0y_1$ be the unique hyperedge in $\CN_1$ containing $y_0y_1$.
Then $c(x)=2$ and $x\in L_i\cup L_{i+1}$. Let $p\in \{0,1\}$ such that $p=1$ iff $x\in L_{i+1}$.
Let $X$ denote the unique path in $T^*$ from $r^*$ to $x$. For each $\ell\in [k]$, let $Y_\ell$ denote the unique path
in $T^*$ from $r^*$ to $y_\ell$. By the definition of $\CN_1$, the path $P_\ell=xy_1\dots y_\ell$ is an extendable path
of length $\ell$ in $\partial \CG$.  Also $V(X)\cap V(P_\ell)=\{x\}$ since $x$ is the only vertex with color $2$ on $P_\ell$.
By definition of $Y_\ell$ and $P_\ell$, $V(Y_\ell)\cap V(P_\ell)=\{y_\ell\}$. So, $V(X\cup Y_\ell)\cap V(P)=\{x,y_\ell\}$. By Claim 1,
$X\cup Y_\ell\cup P_\ell$ is an extendable cycle of length $2(i-i')+p+\ell$ in $\partial \CH$, where $p\in \{0,1\}$.
Since this holds for each $\ell\in [k]$, $\CH$ contains Berge cycles of $k$ consecutive lengths,
the shortest of which has length at most $2i+2$, a contradiction.
Hence, $|\CN_1|\leq \frac{k-1}{2}n_i$. By symmetry, we have $|\CN_2|\leq \frac{k-1}{2} n_i$.
Therefore,
\begin{equation} \label{N1-bound}
|\CN_1|+|\CN_2|\leq  (k-1)n_i
\end{equation}

Now, consider $\CN_3$. By definition, $\CN_3$ is $3$-partite whose three parts are contained in $S_1$, $S_2$, and $L_{i+1}$, respectively.
Suppose $d(\CN_3)>\frac{3k}{2}$. Then by Lemma \ref{prop:lin3par}, there exists $x\in S_1\cap \CV(\CN_3)$ such that for each
$\ell\in [k]$, there is an extendable path $P_\ell$ in $\partial \CN_3$ from $x$ to a vertex $y_\ell$ in $S_2\cap \CV(\CN_3)$. Let $X$ denote the
unique path in $T^*$ from $r^*$ to $x$ and for each $\ell\in [k]$ let $Y_\ell$ denote the unique path in $T^*$ from $r^*$ to $y_\ell$.
Like before, $X\cup Y_\ell\cup P_\ell$ is an extendable cycle of length $2(i-i')+\ell$ in $\partial\CH$. Thus $\CH$ contains Berge cycles of $k$
consecutive lengths, the shortest of which has length at most $2i$, a contradiction. Hence
\begin{equation}\label{N3-bound}
|\CN_3|\leq \frac{k}{2}(n_i+n_{i+1}).
\end{equation}
Claim 3 now follows from \eqref{N1-bound} and \eqref{N3-bound}.
\qed

\medskip

By Claims 2 and 3, $|\CG|\leq (\frac{7k}{2}+1)n_i+(\frac{5k}{2}+2) n_{i+1}$.
Since this holds for each component of $\CB_i\cup \CC_i$, it follows $|\CB_i\cup \CC_i|\leq (\frac{7k}{2}+1)|L_i|+(\frac{5k}{2}+2)|L_{i+1}|$. This proves Lemma \ref{twice-radius-BC}.
\end{proof}

\noindent {\bf Remark.}
In the proof of Claim 2 of Lemma \ref{twice-radius-BC}, one could also use Lemma \ref{special-paths} instead
of Lemma \ref{special-paths-hyper} and get slightly better constants.  However, there are more subtleties to address
if one were to use Lemma \ref{special-paths} and also we want to demonstrate the use of Lemma \ref{special-paths-hyper}
since the lemma will be useful in the study of linear cycles.

\begin{lemma}\label{twice-radius-removal}
Let $\CH$ be a  linear $3$-graph.
Let $T$ be any maximal extendable skeleton in $\CH$ rooted at some vertex $r$ and with height $h$.
If the number of hyperedges in $\CH$ that contain some vertex in $V(T)$ is at least
$7(k+1)|V(T)|$ then $\CH$ contains Berge cycles of $k$ consecutive lengths,
the shortest of which is at most $2h+2$.
\end{lemma}
\begin{proof}
Suppose $\CH$ does not contain such Berge cycles, we derive a contradiction.
Define $\CE_T$ and $\CE'_T$ as in \eqref{ET} and \eqref{ET-prime}.
For each $i\geq 0$, let $L_i=\{v\in V(T): d_T(r,v)=i\}$.
Let $\CA_i=\{e\in \CE(\CH): |e\cap L_i|=2, |L_{i-1}|=1\}$,
$\CB_i=\{e\in \CH: |e\cap L_i|=3\}$, and $\CC_i=\{e\in \CH: |e\cap L_i|=2, |e\cap L_{i+1}|=1\}$.
By Proposition \ref{level-containment}, $\CE'_T\subseteq \bigcup_{i=1}^h (\CA_i\cup \CB_i\cup \CC_i)$.
By Lemma \ref{twice-radius-A} and \ref{twice-radius-BC}, for each $i\in [h]$,
$|\CA_i\cup \CB_i\cup \CC_i|\leq (\frac{9}{2}k+3)|L_i|+(\frac{5}{2}k+2)|L_{i+1}|$.
Hence
$$|\CE'_T|\leq (\frac{9}{2}k+3)\sum_{i=1}^h |L_i| +(\frac{5}{2}k+2)\sum_{i=1}^{h-1}|L_{i+1}|
\leq (7k+5)|V(T)|.$$
On the other hand, $|\CE_T|= |V(T)|-1$. So the number of hyperedges in $\CH$ with at least one vertex in $V(T)$
is less than $7(k+1)|V(T)|$, a contradiction.
\end{proof}

\noindent{\bf Remark.}
Note that in Lemma \ref{twice-radius-removal}, we not only find Berge cycles of  $k$ consecutive lengths,
but also ensure that the shortest length is no more than twice the height of any maximal extendable skeleton.
This extra condition on the shortest length will be useful for studying the Tur\'an number of a Berge cycle of fixed length.
See Section \ref{turan} for detailed discussions.
If we do not impose any condition on the shortest length of our cycles, then we can both improve the
bounds and simplify the proofs of Lemma \ref{twice-radius-A} and Lemma \ref{twice-radius-BC}.

\medskip

We now prove the following theorem, which perhaps is the corner stone of this paper.

\begin{theorem} \label{theorem-twice-radius}
Let $\CH$ be a  linear $3$-graph with $d(\CH)\geq 21(k+1)$.
Then $\CH$ contains Berge cycles of $k$ consecutive lengths.
\end{theorem}
\begin{proof}
By our assumption $|\CH|\geq 7(k+1)|\CV(\CH)|$.
Let $T$ be any maximal extendable skeleton in $\CH$. If $\CH$ does not
contain Berge cycles of $k$ consecutive lengths, then by Lemma \ref{twice-radius-removal},
there are fewer than $7(k+1)|V(T)|$ hyperedges in $\CH$ that contain some vertex in $V(T)$.
Let us delete $V(T)$ and hyperedges that contain vertices in $V(T)$. Denote the remaining
hypergraph by $\CH'$. We can repeat the process until we either find Berge cycles of $k$ consecutive
lengths or we run out of hyperedges. Since $|\CH|\geq 7(k+1)|\CV(\CH)|$ and we lose fewer than $7(k+1)$
hyperedges per vertex we delete, we never run out of hyperedges. So $\CH$ must contain Berge cycles of
$k$ consecutive lengths.
\end{proof}

Now we can derive Theorem \ref{Thm:Li3Berge} promptly from Theorem \ref{theorem-twice-radius}.

\medskip

{\noindent \bf Proof of Theorem \ref{Thm:Li3Berge}:}
Let $\CH$ be a linear $r$-graph with average degree at least $7(k+1)r$. Define $\CG$ to be the $3$-graph with $\CV(\CG)=\CV(\CH)$
by taking as hyperedges a $3$-subset from each edge $e$ of $\CH$. Since $\CH$ is linear, the $3$-subsets are all distinct.
So $|\CG|=|\CH|\geq 7(k+1)|\CV(\CH)|$. Also, $\CG$ is linear. By Theorem \ref{theorem-twice-radius},
$\CG$ contains Berge cycles of $k$ consecutive lengths. Since all edges of $\CG$ extend to distinct edges of $\CH$,
these Berge cycles in $\CG$ extend to Berge cycles of $k$ consecutive lengths in $\CH$.
\qed

%%%%%%%%%%%%%%%%%%%%%%%%%%%%%%%%%%%%

%%%%%%%%%%%%%%%%%%%%%%%%%%%%%%%%%%%%
\section{The proof of Theorem \ref{Thm:main}}

In this section we prove Theorem \ref{Thm:main}, by essentially reducing it to
Theorem \ref{theorem-twice-radius}. We start by providing some useful lemmas.
The following lemma will be important for our reduction. It may be viewed as a very special
case of the delta-system lemma, introduced by Deza, Erd\H{o}s, and Frankl \cite{DEF}.

\begin{lemma}\label{delta-system}
Any $r$-graph $\CH$ has either a subhypergraph $\CH'$ with $\delta_{r-1}(\CH')\ge k+1$,
or a subhypergraph $\CH''$ with $|\CH''|\geq |\CH|/k$ in which each hyperedge contains an $(r-1)$-set with co-degree $1$ in $\CH''$.
In particular, in the latter case, there is an extendable $(r-1)$-graph $\CG\subseteq \partial_{k-1}(\CH)$ with $|\CG|\ge |\CH|/k$.
\end{lemma}
\begin{proof}
We apply the following greedy algorithm for $\CH$.
Initially, set $\CH'=\CH$ and $\CH''=\emptyset$.
If there exist an edge $e\in \CE(\CH')$ and a $(k-1)$-subset $e'\subseteq e$ such that $d_{\CH'}(e')\le k$,
then we place one hyperedge containing $e'$ in $\CH''$ and delete all hyperedges in $\CH'$ containing $e'$.
We continue this until there is no such pair $(e,e')$.
If $\CH'$ is nonempty, then we are done as clearly $\delta_{r-1}(\CH')\ge k+1$.
Hence, $\CH'$ is empty. Then $\CH''$ satisfies that $|\CH''|\geq |\CH|/k$ and that each hyperedge contains
an $(r-1)$-subset with co-degree $1$ in $\CH''$.
For the second statement, form $\CG$  by selecting an $(r-1)$-set with co-degree 1 from each hypergraph of $\CH''$.
Clearly, $|\CG|=|\CH''|$ and $\CG$ is extendable.
\end{proof}

In the next two lemmas, we show that if an $r$-graph $\CH$ has large $\delta_{r-1}(\CH)$,
then there exist many Berge cycles of consecutive lengths.
An $r$-graph $\CP$ is a {\it tight path of length $m$},
if it consists of $m$ edges $e_1,\dots,e_m$ and $m+r-1$ vertices $v_1,\dots,v_{m+r-1}$ such that each $e_i=\{v_i,\dots,v_{i+r-1}\}$.

\begin{lemma}\label{tight-path}
Let $r\ge 3$ and $\CP$ be an $r$-graph. If $\CP$ is a tight path of length $m+1$,
then $\CP$ contains Berge cycles of all lengths in $\{3,\dots,m\}$.
\end{lemma}

\begin{proof}
Let $e_1,\dots,e_{m+1}$ be all edges in $\CP$ such that each $e_i=\{i,\dots,i+r-1\}$.
Let $f_i=\{i,i+1,i+2\}$. Since each $f_i$ extends to $e_i$,
it suffices to find Berge cycles in the tight path $\CP'=\{f_1,...,f_{m+1}\}$ (i.e., it suffices to consider 3-graphs).
For even $t\le m+1$, consider the following 2-cycle with spine $2,4,...,t,t-1,t-3,...,3,2$.
This 2-cycle can extend to a Berge cycle of length $t-1$ in $\CP'$ with edges $f_2,f_4,...,f_{t-2},f_{t-1},f_{t-3},...,f_{3},f_1.$
There edges cover pairs $24, 46,...,(t-2)t,t(t-1),(t-1)(t-3),...,53,32$, respectively.
For odd $t\le m+1$, similarly consider the following 2-cycle with spine $2,4,...,t-3,t-1,t,t-2,...,3,2$.
This 2-cycle can extend to a Berge cycle of length $t-1$ in $\CP'$ with edges $f_2,f_4,...,f_{t-3},f_{t-1},f_{t-2},f_{t-4},...,f_{3},f_1.$
There edges cover pairs $24, 46,...,(t-3)(t-1),(t-1)t,t(t-2),(t-2)(t-4),...,53,32$, respectively.
Hence, there exist Berge cycles of all lengths in $\{3,\dots,m\}$ in $\CP'$ (and thus in $\CP$).
\end{proof}

\begin{lemma}\label{berge-cycle}
Let $r\geq 3$ and $\CH$ be an $r$-graph with $\delta_{r-1}(\CH)\geq k+1$.
Then $\CH$ contains Berge cycles of all lengths in $\{3,4,...,k+2\}$.
\end{lemma}

\begin{proof}
Let $\CP$ be a longest tight path in $\CH$, say of length $m$.
Let $e_1,\dots,e_{m}$ be all edges in $\CP$ such that each $e_i=\{v_i,\dots,v_{i+r-1}\}$.
Let $S=\{v_{m+1},...,v_{m+r-1}\}$ be a subset of vertices with size $r-1$.
All edges $f$ of $\CH$ containing $S$ satisfy that $f\backslash S\subseteq \{v_1,...,v_m\}$,
as otherwise $\CP\cup f$ would be a longer tight path than $\CP$, contradicting the choice of $\CP$.
This also shows that $m\ge d_{\CH}(S)\ge d_{r-1}(\CH)\geq k+1$.
By Lemma \ref{tight-path}, $\CP$ contains Berge cycles of all lengths in $\{3,\dots,m-1\}$.
Hence, we may assume that $k+1\le m\le k+2$ and in particular there exist Berge cycles of all lengths in $\{3,4,...,k\}$ in $\CH$.

Suppose $m=k+1$. It is clear that $d_\CH(S)=k+1$ and all $k+1$ edges $f_1,...,f_{k+1}$
in $\CH$ containing $S$ are such that $f_i=S\cup \{v_i\}$. Then there exist a Berge cycle
$\{e_1,...,e_{k+1},f_1\}$ of length $k+2$ with spine $v_1,v_2,...,v_{k+1},v_{k+2}$ and a Berge cycle
$\{e_2,...,e_{k+1},f_2\}$ of length $k+1$ with spine $v_2,...,v_{k+1},v_{k+2}$.
So $\CH$ contains Berge cycles of all lengths in $\{3,4,...,k+2\}$.

Therefore, we have $m=k+2$. There are at least $k+1$ edges $f_1,...,f_{k+1}$
in $\CH$ containing $S$ such that $f_i\backslash S\subseteq \{v_1,...,v_{k+2}\}$.
So there is some $i\in [k+2]$ such that for every $j\in [k+2]\backslash \{i\}$, $S\cup \{v_j\}\in \CE(\CH)$.
Let $j_0, j_1$ be the first and second integers in $[k+2]\backslash \{i\}$.
Then there exist a Berge cycle $\{e_j: j\in [k+2]\backslash \{i\}\}\cup (S\cup \{v_{j_0}\})$ of length $k+2$
with spine $\{v_j: j\in [k+3]\backslash \{i\}\}$, and a Berge cycle $\{e_j: j\in [k+2]\backslash \{i,j_0\}\}\cup (S\cup \{v_{j_1}\})$ of length $k+1$
with spine $\{v_j: j\in [k+3]\backslash \{i,j_0\}\}$. This finishes the proof.
\end{proof}

We are now ready to prove Theorem \ref{Thm:main}. We will use induction on $r$. The following theorem
forms the basis step.

\begin{theorem}\label{Thm:main3gr}
Let $\CG$ be a $3$-graph with $d(\CG)\geq 105k^2+63k$.
Then $\CG$ contains Berge cycles of $k$ consecutive lengths.
\end{theorem}
\begin{proof}
By Lemma \ref{delta-system}, $\CG$ has either a subhypergraph $\CG'$ with $\delta_2(\CG')\geq k+1$,
or a subhypergraph $\CG''$ with $|\CG''|\geq |\CG|/k$  in which each hyperedge contains a pair that has co-degree $1$ in $\CG''$.
In the former case, by Lemma \ref{berge-cycle}, $\CG'$ contains Berge cycles of $k$ consecutive lengths and we are done.
Hence, we may assume the latter case. For each hyperedge in $\CG''$ let us mark a pair in it that has co-degree $1$.
By our assumption, each hyperedge in $\CG''$ has a marked pair.

Let us call a pair $uv$ a {\it high pair} if its co-degree in $\CG''$ is at least $3$ and a {\it low pair} otherwise.
Let $\CG_1$ consist of all the hyperedges in $\CG''$ that contain a high pair and
$\CG_2$ consist of all the other hyperedges in $\CG''$.
Since $d(\CG'')\geq d(\CG)/k\geq 105k+63$, one of the following two cases applies.

\medskip

{\bf Case 1.} $d(\CG_1)\geq 42k$.

\medskip

Let $S$ be a random set of vertices with each vertex of $\CG_1$ selected independently with probability $\frac{2}{3}$.
For each hyperedge in $\CG_1$, call it {\it good} for $S$ if the two vertices in its marked pair are both in $S$ and the third
vertex is not in $S$. The probability of a hyperedge being good is $\frac{4}{27}$.
So there exists a set $S$ for which at least $\frac{4}{27}|\CG_1|$ of
the hyperedges are good. Fix such a set $S$ and let $\CG_1^*$ consist of all the good hyperedges of $\CG_1$. By our assumption
$$d(\CG^*_1)\geq \frac{4}{27} d(\CG_1)\geq 6k.$$
Let $G_1=\{e\cap S: e\in \CG^*_1\}$. Then $G_1$ is a $2$-graph and
there is a bijection between edges in $G_1$ and hyperedges in $\CG^*_1$. In particular, $|G_1|=|\CG^*_1|$.
Hence $d(G_1)\geq \frac{2}{3} d(\CG^*_1)\geq 4k$. By Theorem \ref{Jacques-even-cycles},
 $G_1$ contains
cycles of $k/2$ consecutive even lengths. To complete this case, observe that if $C$ is a cycle of
length $\ell$ in $G_1$, then $\CG$ contains a Berge cycle of length $\ell$ and a Berge cycle of length $\ell+1$.
Indeed, suppose $C=u_1u_2\dots u_\ell u_1$. By our definition, all edges on $C$ extend to
different hyperdges in $\CG$. So we obtain a Berge cycle of length $\ell$. Let $u_1u_2w$ be the unique
hyperedge in $\CG^*_1$. By definition, $w\notin S$. Also, at least one pair in $u_1u_2w$ is a high pair.
Since $u_1u_2$ is a marked pair and has co-degree $1$, either $u_1w$ or $u_2w$ is a high pair.
By symmetry suppose $u_1w$ is a high pair. Since $u_1w$ has co-degree at least $3$ in $\CG''$
there is a hyperedge $u_1wz$ where $z\notin \{u_2, u_\ell\}$. It is easy to see that $u_1wu_2\cdots u_\ell u_1$
is an extendable cycle in $\partial(\CG'')$ of length $\ell+1$. So $\CG$ contains a Berge cycle of length $\ell+1$.

\medskip

{\bf Case 2.} $d(\CG_2)\geq 63(k+1)$.

\medskip

By our assumption, for each hyperedge in $\CG_2$, one of its pairs have co-degree $1$ and the other two have
co-degree at most $2$. Define an auxiliary graph $L$ whose vertices are hyperedges in $\CG_2$ such
that two vertices in $L$ are adjacent if the corresponding hyperedges in $\CG_2$ share a pair. Then $L$ has maximum degree
at most $2$ and thus has an independent set of size at least $n(L)/3$.
Therefore, there is a linear subhypergraph $\CG^*_2$ of $\CG_2$
with $|\CG^*_2|\geq \frac{1}{3}|\CG_2|$. Hence $d(\CG^*_2)\geq 21(k+1)$.
By Theorem \ref{theorem-twice-radius}, $\CG^*_2$ contains Berge cycles of $k$ consecutive lengths.
\end{proof}

We now prove Theorem \ref{Thm:main}.

\bigskip

{\noindent\bf Proof of Theorem \ref{Thm:main}.}
We prove by induction on $r\geq 3$ that every $r$-graph $\CH$ with $d(\CH)\ge r\cdot (35k^{r-1}+21k^{r-2})$
contains Berge cycles of $k$ consecutive lengths.
Theorem \ref{Thm:main3gr} forms the basis step. For the induction step, let $r\geq 4$.

Assume the claim holds for $(r-1)$-graphs.
Let $\CH$ be an $r$-graph with $d(\CH)\ge r\cdot (35k^{r-1}+21k^{r-2})$.
By Lemma \ref{delta-system},
either there exists a subhypergraph $\CH'\subseteq \CH$ with $\delta_{r-1}(\CH')\ge k+1$,
or there exists an extendable $(r-1)$-graph $\CG\subseteq \partial_{r-1}(\CH)$ such that $|\CG|\geq |\CH|/k$.
In the former case, by Lemma \ref{berge-cycle}, we can find Berge cycles of lengths in $\{3,4,...,k+2\}$ in $\CH$
and we are done. So assume the latter case. Then $d(\CG)\ge \frac{r-1}{r}\cdot \frac{d(\CH)}{k}\geq (r-1)\cdot (35k^{r-2}+21k^{r-3})$.
By induction, $\CG$ contains Berge cycles of $k$ consecutive lengths.
Because $\CG$ is extendable, $\CH$ also contains Berge cycles of the same $k$ consecutive lengths.
\qed

\bigskip

We then have the following corollary.

\begin{corollary}\label{cor:edges}
There exists an absolute constant $c>0$ such that the following holds for all $r\ge 3$.
Any $n$-vertex $r$-graph $\CH$ with at least $ck^{r-1}n$ edges contains Berge cycles of $k$ consecutive lengths.
\end{corollary}

%%%%%%%%%%%%%%%%%%%%%%%%%%%%%%%%%%%%%

%%%%%%%%%%%%%%%%%%%%%%%%%%%%%%%%%%%%%

\section{Related Tur\'an type results}  \label{turan}

\subsection{Cycles of consecutive even lengths in graphs}

\medskip
Following arguments along the line of Lemma \ref{twice-radius-BC} (i.e.,
to define {\it monochromatic} and {\it non-monochromatic} edges and then apply Lemma \ref{special-paths}), we
can readily prove the following slightly weaker version of Verstra\"ete's theorem.

\medskip

\begin{proposition} \label{consecutive-even}
Let $G$ be a bipartite graph with average degree at least $6k$ and radius $h$.
Then $G$ contains cycles of $k$ consecutive even lengths. Further, the shortest
of these cycles has length at most $2h$.
\end{proposition}

This provides a first proof which does not use the $(A,B)$-path lemma (Lemma \ref{BS}).
This also gives yet another proof of the theorem of Bondy and Simonovits on $ex(n,C_{2k})$ without using
either the $(A,B)$-path lemma or the Faudree-Simonovits blowup method.
As Lemma \ref{special-paths} (and its hypergraph extension) can be easily adapted,
we anticipate this new method will find further applications in Tur\'an type extremal problems on cycles
in graphs or hypergraphs.

\subsection{Berge cycles of prescribed consecutive lengths}
Let $\CB C_\ell$ denote the family of $r$-graphs consisting of all Berge cycles of length $\ell$.
Let $r\ge 3$ and $\CH$ be an $r$-graph with $n$ vertices.
Corollary \ref{cor:edges} shows that if $|H|\ge \Omega(k^{r-1}n)$, then $\CH$ contains Berge
cycles of $k$ consecutive lengths.
If in addition to find Berge cycles of $k$ consecutive lengths one also wants to control the lengths to not be large,
i.e., the maximum length is no more than $k+p$, then how many edges in an $r$-graph will suffice?
In this subsection we provide an answer to this question (see Theorem \ref{Thm:length-control}).

Using Lemmas \ref{twice-radius-A} and \ref{twice-radius-BC}, we can prove the following theorem for linear 3-graphs.
Its proof uses similar arguments as the ones in \cite{FO} by F\"uredi and \"Ozkahya, who proved that
any $n$-vertex $\CB C_{2k+1}$-free linear 3-graph has at most $2k\cdot n^{1+1/k}+9k\cdot n$ edges.
However their proof method was not designed for finding Berge cycles of consecutive lengths.

\begin{theorem} \label{cycles-in-linear-dense}
Let $h,k\geq 2$. Every $n$-vertex linear $3$-graph $\CH$ with $|\CH|\geq 18k n^{1+1/h}+42kn$ contains Berge cycles of $k$ consecutive lengths,
the shortest of which has length at most $2h$.
\end{theorem}

\begin{proof}
By Lemma \ref{prop:ave->min}, there exists a subhypergraph $\CH'$ of $\CH$ with $\delta(\CH')\ge 18kn^{1/h}+42k$.
Let $T$ be a maximal extendable skeleton in $\CH'$ rooted at some vertex $r$.
For each $i\geq 0$, let $L_i=\{v\in V(T): d_T(r,v)=i\}$, $\CA_i=\{e\in \CH': |e\cap L_i|=2, |L_{i-1}|=1\}$,
$\CB_i=\{e\in \CH': |e\cap L_i|=2, |e\cap L_{i+1}|=1\}$, and $\CC_i=\{e\in \CH': e\subseteq L_i\}$.
By Lemmas \ref{twice-radius-A} and \ref{twice-radius-BC}, for any $i\le h-1$ we may assume
(being quite generous for the sake of simplicity) that
\begin{equation} \label{ABC}
|\CA_i|\leq 2k|L_i|  \mbox{ and }|\CB_i\cup \CC_i|\leq 4k|L_i|+4k|L_{i+1}|.
\end{equation}

We prove by induction on $1\le i\le h-1$ that $|L_i|\geq |L_{i-1}|\cdot n^{1/h}$.
The base case $i=1$ follows by the facts that $\CH'$ is linear and $\delta(\CH')\ge 18kn^{1/h}$.
Now suppose that it holds for $i\le h-2$. We consider the number $m$ of edges intersecting $L_{i}$.
We have $m=|\CA_{i}|+|\CB_{i}|+|\CC_{i}|+|\CC_{i-1}|+|\CA_{i+1}|$. By \eqref{ABC},
\begin{align*}
m&\le 2k|L_i|+4k|L_i|+4k|L_{i+1}|+4k|L_{i-1}|+4k|L_i|+2k|L_{i+1}|\\
&=4k|L_{i-1}|+10k|L_i|+6k|L_{i+1}|\leq 14k|L_i|+6k|L_{i+1}|.
\end{align*}
On the other hand, $m\geq \frac13\cdot \sum_{v\in L_i} d_{\CH'}(v)\ge |L_i|\cdot (6kn^{1/h}+14k)$.
Combining the above inequalities, it follows that $|L_{i+1}|\geq |L_i|\cdot n^{1/h}$.
Therefore, $|L_h|\geq n$. This contradiction completes the proof.
\end{proof}
We point out that just like in the proof of Theorem \ref{Thm:Li3Berge}, one can establish
a similar statement for linear $r$-graphs for all $r\geq 3$ (by reducing them to linear $3$-graphs).

\medskip

\noindent{\bf Remark.} Observe that the proof of Theorem \ref{cycles-in-linear-dense} in fact yields
the following more general statement:  If $\CH$ is an $n$-vertex  linear $3$-graph with average degree
$d\geq 45k$, then $\CH$ contains Berge cycles of $k$ consecutive lengths, the shortest
of which has length at most $O(\log_{d/k} n)$.

\medskip

Using the reduction lemmas in Section 4, along the same lines as in the proof of Theorem \ref{Thm:main} (and Theorem \ref{Thm:main3gr}),
we also can obtain the following result from Theorem \ref{cycles-in-linear-dense}. We omit the details.

\begin{theorem}\label{Thm:length-control}
There exists an absolute constant $c>0$ such that the following holds for all $h,k\ge 2$ and $r\ge 3$.
Every $n$-vertex $r$-graph $\CH$ with at least $ck^{r-1}n^{1+1/h}$ edges contains Berge cycles of $k$ consecutive lengths,
the shortest of which has length at most $2h$.
\end{theorem}

When choosing $k=2h$, this may be viewed as an unification for the results on $ex_r(n,\CB C_{2h})$ and $ex_r(n,\CB C_{2h+1})$.

%%%%%%%%%%%%

\subsection{Tur\'an numbers of Berge cycles in $r$-graphs}
In this subsection we investigate the upper bounds of Tur\'an numbers $ex_r(n,\CB C_\ell)$ for $r\ge 3$.
We start by mentioning the Tur\'an numbers of even cycles in the graph case.
A classic theorem of Bondy and Simonovits \cite{BS74} shows that $ex(n,C_{2k})\le 100k\cdot n^{1+1/k}$,
and this bound was improved by several authors in \cite{V00,Pik,BJ}.
The current best known upper bound is the following one obtained by Bukh and Jiang \cite{BJ}
\begin{equation}\label{equ:exC2k}
ex(n,C_{2k})\le 80\sqrt{k\log k}\cdot n^{1+1/k}+O(n).
\end{equation}

For 3-graphs, Gy\H{o}ri and Lemons proved that $ex_3(n,\CB C_{2k+1})\le O(k^4)\cdot n^{1+1/k}$ in \cite{GL-3uniform}
and that $ex_3(n, \CB C_{2k})\le O(k^2)\cdot ex(n,C_{2k})$ in \cite{GL}.
F\"uredi and \"Ozkahya \cite{FO} improved this by showing that
$$ex_3(n,\CB C_{2k+p})\le O(k)\cdot ex(n,C_{2k})+ 12p\cdot ex_3^{lin}(n,\CB C_{2k+1}) \text{~~~for every ~} p\in \{0,1\},$$
where $ex_r^{lin}(n,\mathcal{F})$ denotes the maximum number of hypedges in an $n$-vertex $\mathcal{F}$-free linear $r$-graph
and it is also proved in \cite{FO} that $ex_3^{lin}(n,\CB C_{2k+1})\le 2k\cdot n^{1+1/k}+9k\cdot n$.
(See \cite{AS} for related problems.)
In view of \eqref{equ:exC2k}, one can obtain
\begin{equation}\label{equ:ex3C2k+1}
ex_3(n,\CB C_{2k+p})\le O(k\sqrt{k\log k})\cdot n^{1+1/k} \text{~~~for every ~} p\in \{0,1\}.
\end{equation}

In \cite{GL}, Gy\H{o}ri and Lemons also showed for general $r$-graphs, where $r\ge 4$, that
\begin{equation}\label{equ:exrC2k+1}
ex_r(n,\CB C_{2k+1})\le O(k^{r-2})\cdot ex_3(n,\CB C_{2k+1}),
\end{equation}
\begin{equation}\label{equ:exrC2k}
ex_r(n,\CB C_{2k})\le O(k^{r-1})\cdot ex(n,C_{2k}).
\end{equation}
Using the lemmas in Section 4, one can also derive some Tur\'an type results on Berge cycles,
which improve the above inequalities \eqref{equ:exrC2k+1} and \eqref{equ:exrC2k} by an $\Omega(k)$ factor.

\begin{proposition}\label{prop:ex(Berg)}
For all $r\ge 4$, it holds that
$$ex_r(n, \mathcal{B}C_{2k+1})\le O(k^{r-3})\cdot ex_3(n,\mathcal{B}C_{2k+1})$$
$$ex_r(n, \mathcal{B}C_{2k})\le O(k^{r-2})\cdot ex(n,C_{2k}).$$
Therefore for any $p\in \{0,1\}$, $$ex_r(n,\CB C_{2k+p})\le O(k^{r-2}\sqrt{k\log k})\cdot n^{1+1/k}.$$
\end{proposition}
The proof of Proposition \ref{prop:ex(Berg)} follows easily from
Lemmas \ref{delta-system} and \ref{berge-cycle}. We omit the details.
We mention another related result. In \cite{GL}, the following result was also proved for non-uniform hypergraphs: for any $p\in \{0,1\}$, if $\CH$ is a multi-hypergraphs on $n$ vertices with all of its hyperedges of size at least $4k^2$ and containing no Berge cycle of length $2k+p$, then
\begin{equation}\label{equ:non-uniform}
\sum_{e\in \CE(\CH)} |e|\le (16k^6+8k^2)\cdot n^{1+1/k}+(16k^7+32k^6+16k^5)\cdot n.
\end{equation}

%%%%%

\subsection{Asymmetric Tur\'an numbers of even cycles in graphs}
Let the {\it Zarankiewicz number} $z(m,n,C_{2k})$ of the even cycle $C_{2k}$ to be the maximum number
of edges in a $C_{2k}$-free bipartite graph with two parts of sizes $m$ and $n$.
An upper bound was proved by Naor and Verstra\"ete \cite{NV} that for $m\le n$ and $k\ge 2$,
\begin{equation}\label{equ:z(C2k)-0}
z(m,n,C_{2k})\le \left\{\begin{array}{ll}
    (2k-3)\cdot [(mn)^{\frac{k+1}{2k}}+m+n] & \text{if } k \text{ is odd}, \\
    (2k-3)\cdot [m^{\frac{k+2}{2k}}n^{\frac{1}{2}}+m+n] & \text{if } k \text{ is even}.
  \end{array}\right.
\end{equation}
In this subsection, we consider a different form of upper bounds about $z(m,n,C_{2k})$.

Erd\H{o}s, S\'ark\"ozy and S\'os \cite{ESS} conjectured that $z(m,n,C_6)<2n+c(nm)^{2/3}$ for some constant $c>0$.
A weaker version of this conjecture was obtained by S\'ark\"ozy in \cite{Sa}.
Gy\H{o}ri \cite{G97} proved a general result: there exists some $c_k>0$ such that for $n\ge m^2$,
\begin{equation}\label{equ:z(C2k)-1}
z(m,n,C_{2k})\le (k-1)n+c_k\cdot m^2.
\end{equation}
The first term $(k-1)n$ is sharp (at least in a sense) by considering the complete bipartite graph $K_{k-1,n}$;
and when $n=\Omega(m^2)$ this function becomes linear in $n$.
Some related results also can be found in \cite{BGMV,G06}.

The following upper bound of $z(m,n,C_{2k})$ can be derived from the Tur\'an numbers of Berge cycles in hypergraphs,
which is stronger than \eqref{equ:z(C2k)-1}.

\begin{proposition}\label{prop:z(C2k)}
There exists a constant $d_k>0$ such that for any positive integers $n,m$,
$$z(m,n,C_{2k})\le (k-1)n+d_k\cdot m^{1+1/{\lfloor k/2\rfloor}}.$$
\end{proposition}

\begin{proof}
Let $G$ be any bipartite $C_{2k}$-free graph with two parts $A$ and $B$, where $|A|=m$ and $|B|=n$.
Define $\CH_0$ and $\CH_i$ for every $k\le i< 4k^2$ to be multi-hypergraphs with the vertex-set $A$ such that
$\CE(\CH_0)=\{N_G(u): d_G(u)\ge 4k^2 \text{ for } u\in B\}$ and $\CE(\CH_i)=\{N_G(u): d_G(u)=i \text{ for } u\in B\}$.
Then $$e(G)=\sum_{u\in B} d_G(u)\le (k-1)n+\sum_{e\in \CE(\CH_0)}|e|+ \sum_{k\le i< 4k^2}\sum_{e\in \CE(\CH_i)}|e|.$$
Observe that $\CH_0$ and all $\CH_i$ do not contain Berge cycle of length $k$, as otherwise it will give a $C_{2k}$ in $G$.
By the equation \eqref{equ:non-uniform}, one can get that $$\sum_{e\in \CE(\CH_0)}|e|\le O(k^6)\cdot m^{1+1/{\lfloor k/2\rfloor}}.$$
Consider the multi-hypergraph $\CH_i$, where $k\le i< 4k^2$.
It is easy to see that there are at most $k-1$ hyperedges which are identical (otherwise one can form a Berge cycle of length $k$ easily).
Thus, there exists a simple $\CB C_k$-free $i$-graph $\CH_i'\subseteq \CH_i$ such that $|\CH_i'|\ge |\CH_i|/k$.
By Proposition \ref{prop:ex(Berg)} (or Theorem \ref{Thm:length-control}), $|\CH_i|\le k\cdot |\CH_i'|\le O(k^i)\cdot m^{1+1/{\lfloor k/2\rfloor}}$.
Combining the above inequalities, one can obtain that
$$e(G)\le (k-1)n+O(k^6)\cdot m^{1+1/{\lfloor k/2\rfloor}}+\sum_{k\le i< 4k^2} i\cdot |\CH_i|\le (k-1)n+d_k\cdot m^{1+1/{\lfloor k/2\rfloor}},$$
where $d_k=O(k^{4k^2+1})$. This finishes the proof.
\end{proof}

One may compare this proposition with \eqref{equ:z(C2k)-0}.
In the range $m^{1+1/{\lfloor k/2\rfloor}}/n\to 0$, interestingly both upper bounds becomes linear in $n$.
To be precise, \eqref{equ:z(C2k)-0} gives that $z(m,n,C_{2k})\le (2k-3+o(1))\cdot n$,
and this proposition yields $z(m,n,C_{2k})\le (k-1+o(1))\cdot n$, which is nearly tight.

%%%%%%%%%%%%%%%%%%%%%%%%%%%%%%%%%%%%

\section{Future work}

\subsection{Finding tight conditions for Berge cycles of consecutive lengths}

It will be interesting to completely solve the problem on Berge cycles  of consecutive lengths.

\begin{problem}
For $r\ge 3$, find the minimum $f_r(k)$ such that any $r$-graph with average-degree $f_r(k)$ (or minimum-degree) contains Berge cycles of $k$ consecutive lengths.
\end{problem}
\noindent The complete $r$-graph on $k+1$ vertices shows that the minimum average-degree $f_r(k)>\binom{k}{r-1}$.
One also can ask for the tight degree condition for the existence of Berge cycles
of $k$ consecutive lengths in linear $r$-graphs.
% In this case we suspect that the Steiner triple system on $k+2$ vertices is extremal.

On a related note, let us note that
Theorem \ref{Thm:main} can be rephrased as the following: every $r$-graph $\CH$ with minimum 1-degree $\delta(\CH)=\Omega(k^{r-1})$
contains Berge cycle of $k$ consecutive lengths. From this, one can promptly obtain an analog for minimum $i$-degree.

\begin{corollary}\label{cor:min-i-degree}
Any $r$-graph $\CH$ with minimum $i$-degree $\delta_i(\CH)=\Omega(k^{r-i})$
contains Berge cycle of $k$ consecutive lengths.
\end{corollary}

\noindent For the proof, it suffices to show that if $\delta_i(\CH)=\Omega(k^{r-i})$ then $\delta_{i-1}(\CH)=\Omega(k^{r-i+1})$.
%%%%%%%%%%%%%%%%%%%%%%%%%%

\subsection{Linear cycles of consecutive lengths and the linear Tur\'an problem}

It is natural to consider the analogous problem for linear cycles, rather than Berge cycles in $r$-uniform hypergraphs.
When the host graph is not required to be linear, the problem is essentially solved for large $n$, due to the
solution to the corresponding Tur\'an problem by F\"uredi and Jiang \cite{FJ} (for $r\geq 5$)
and Kostochka, Mubayi and Verstra\"ete \cite{KMV} (for $r\geq 3$). However, the corresponding problem
for linear cycles of consecutive lengths in linear $r$-uniform hypergraphs is still relatively open.
Collier-Cartaino, Graber and Jiang \cite{CGJ} considered the corresponding linear Tur\'an problem.
Let $ex^{lin}_r(n, C_\ell)$ denote the maximum number of hyperedges in an $n$-vertex linear
$r$-uniform hypergraph that does not contain a linear cycle of length $\ell$.
Extending Bondy-Simonovits \cite{BS74}, they showed that
$ex^{lin}_r(n, C_{2k+p})\leq c_k\cdot n^{1+1/k}$ for some constant $c_k>0$, where $p\in \{0,1\}$.

Using the method developed in this paper together with  ideas from \cite{CGJ},
we can obtain an analogous version of Theorem \ref{Thm:Li3Berge} for linear cycles of consecutive lengths
(which would be a strengthening of Theorem \ref{Thm:Li3Berge} except for the coefficient).
Namely, we can show that every linear $r$-graph with average degree $\Omega(k)$ contains
linear cycles of $k$ consecutive lengths. This can then be used to reduce the coefficient $c_k$ in
the  bound on $ex^{lin}_r(n, C_{2k+p})$. However, due to the additional technicality of that argument, we will leave it
for a forthcoming paper together with other results that we may obtain using our method.

\subsection{Rainbow Tur\'an problem for even cycles}

A problem that is closely related to the linear Tur\'an problem of linear cycles is rainbow Tur\'an problem
for even cycles. For a fixed graph $H$, define the {\it rainbow Tur\'an number} $ex^*(n,H)$ to be the
maximum number of edges in an $n$-vertex graph that has a proper edge-coloring with no rainbow $H$.
Keevash, Mubayi, Sudakov and Verstra\"ete \cite{KMSV} made the following conjecture

\begin{conjecture} {\bf(\cite{KMSV})}
For all $k\geq 2$, $ex^*(n,C_{2k})=O(n^{1+1/k})$.
\end{conjecture}

The conjecture was verified for $k=2,3$ in \cite{KMSV}, but is otherwise still open.
It will be interesting to see if the method developed here can be used to make some progress on the problem.

We direct readers to the recent survey \cite{V16} by Verstra\"ete for various other extremal problems on cycles.

%%%%%%%%%%%%%%%%%%%%%%%%%%%%%%%%%%%%%%%%%%


\begin{thebibliography}{10}
\bibitem{AS} N. Alon and C. Shikhelman,
Many $T$ copies in $H$-free graphs,
\emph{J. Combin. Theory Ser. B}, to appear.
DOI:10.1016/j.jctb.2016.03.004

\bibitem{BGMV} C. Balbuena, P. Garc\'ia-V\'azquez, X. Marcote and J. C. Valenzuela,
Counterexample to a conjecture of Gy\H{o}ri on $C_{2l}$-free bipartite graphs,
\emph{Discrete Math.} 307 (2007), 748--749.


%\bibitem{BM} J. Bondy and U. Murty,
%Graph theory, Graduate Texts in Mathematics,
%244. Springer, New York, 2008. xii+651 pp. ISBN: 978-1-84628-969-9.


%\bibitem{B77}
%B. Bollob\'as, \newblock{Cycles modulo k}, \newblock{\emph{Bull. London Math. Soc.}} \textbf{9} (1977), 97--98.

\bibitem{BG} B. Bollob\'as, E. Gy\"ori, Pentagons vs. triangles, \emph{Discrete Math.} 308 (2008), 4332--4336.

\bibitem{BS74} J. Bondy and M. Simonovits,
Cycles of even length in graphs,
\emph{J. Combin. Theory Ser. B} 16 (1974), 97--105


\bibitem{BV98}
J. A. Bondy and A. Vince,
\newblock{Cycles in a graph whose lengths differ by one or two},
\newblock{\emph{J. Graph Theory}} \textbf{27} (1998), 11--15.


\bibitem{BJ} B. Bukh and Z. Jiang,
A bound on the number of edges in graphs without an even cycle,
\emph{Combin. Probab. Comput.}, to appear, (see also
arXiv:1403.1601v2).

\bibitem{CGJ} C. Collier-Cartaino, N. Graber and T. Jiang,
Linear Tur\'an numbers of $r$-uniform linear cycles and related Ramsey numbers,
\emph{Combin. Probab. Comput.}, to appear, (see also arXiv:1404.5015).


%\bibitem{D10}
%A. Diwan, \newblock{Cycles of even lengths modulo k}, \newblock{\emph{J. Graph Theory}} \textbf{65} (2010), 246--252.

\bibitem{DEF} M. Deza, P. Erd\H{o}s, P. Frankl, Intersection properties of systems of finite sets,
\emph{Proc. London. Math. Soc. (3)} 36 (1978), 369--384.

\bibitem{Erd76}
P. Erd\H{o}s, \newblock{Some recent problems and results in graph theory, combinatorics, and number theory},
\newblock{\emph{Proc. Seventh S-E Conf. Combinatorics, Graph Theory and Computing, Utilitas Math.}}, Winnipeg, 1976, pp 3--14.

\bibitem{EG}
P. Erd\H{o}s, T. Gallai, On maximal paths and circuits of graphs, \emph{Acta. Math. Acad. Sci. Hung.} 10 (1959), 337--356.

\bibitem{EK}
P. Erd\H{o}s, D. Kleitman, \newblock{On coloring graphs to maximize the proportion of multicolored $k$-edges},
\emph{J. Combin. Theory} 5, (1968), 164--169.

\bibitem{ESS} P. Erd\H{o}s, A. S\'ark\"ozy and V.T. S\'os,
On product representation of powers, I,
\emph{European J. Combin.} 16 (1995), 567--588.

\bibitem{ES} P. Erd\H{o}s, M. Simonovits, A limit theorem in graph theory, \emph{Stuidia Sci. Math. Hungar.} 1 (1966), 51--57.

\bibitem{Es-Stone} P. Erd\H{o}s, A.M. Stone  On the structure of linear graphs, \emph{Bulletin of Amer. Math. Soc.} 52 (12) (1946), 1087--1091.

\bibitem{Fan} G. Fan,
Distribution of Cycle Lengths in Graphs,
\emph{J. Combin. Theory Ser. B} 84 (2002), 187--202.

\bibitem{FS} R. Faudree, M. Simonovits, On a class of degenerate extremal graph
problems, \emph{Combinatorica} 3 (1983), 83-93.


\bibitem{FJ} Z. F\"uredi and T. Jiang.
Hypergraph Tur\'an numbers of linear cycles,
\emph{J. Combin. Theory Ser. A.} 123 (2014), 252--270.

\bibitem{FO} Z. F\"uredi and L. \"Ozkahya,
On 3-uniform hypergraphs without a cycle of a given length,
\emph{arXiv:1412.8083v2}


%\bibitem{FS} Z. F\"uredi and M. Simonovits,
%The history of degenerate (bipartite) extremal graph problems, in \emph{Erd\H{o}s Centennial},
%pp. 169--264, Bolyai Soc. Math. Stud., 25, Springer, Berlin, 2013.

\bibitem{G97} E. Gy\H{o}ri,
$C_6$-free bipartite graphs and product representation of squares,
\emph{Discrete Math.} 165/166 (1997), 371--375.



\bibitem{G06} E. Gy\H{o}ri,
Triangle-free hypergraphs,
\emph{Combin. Probab. Comput.} 15 (2006), 185--191.


\bibitem{GL-3uniform} E. Gy\H{o}ri and N. Lemons,
3-uniform hypergraphs avoiding a given odd cycle,
\emph{Combinatorica} 32 (2012), 187--203


\bibitem{GL} E. Gy\H{o}ri and N. Lemons,
Hypergraphs with no cycle of a given length,
\emph{Combin. Probab. Comput.} 21 (2012), 193--201.



\bibitem{HS98}
R. H\"aggkvist and A. Scott,
\newblock{Arithmetic progressions of cycles},
\newblock{\emph{Technical Report}} No. 16 (1998), Matematiska Institutionen, Ume\r{a}Universitet.

\bibitem{KMSV} P. Keevash, D. Mubayi, B. Sudakov, J. Verstra\"ete,
Rainbow Tur\'an problems, \emph{ Combin. Probab. Comput.} 16 (2007), 109--126.



\bibitem{KMV} A. Kostochka, D. Mubayi, J. Verstra\"ete,
Tur\'an problems and shadows I: paths and cycles,
\emph{J. Combin. Th. Ser. A}, to appear.

\bibitem{KSV}
A. Kostochka, B. Sudakov and J. Verstra\"ete,
Cycles in triangle-free graphs of large chromatic number,
\emph{Combinatorica}, to appear. DOI: 10.1007/s00493-015-3262-0


\bibitem{LM} C.-H. Liu and J. Ma,
Cycle lengths and minimum degree of graphs,
\emph{arXiv:1508.07912}


\bibitem{Ma} J. Ma,
Cycles with consecutive odd lengths,
\emph{European J. Combin.} 52(A) (2016), 74--78.

\bibitem{NV} A. Naor and J. Verstra\"ete,
A note on bipartite graphs without $2k$-cycles,
\emph{Combin. Probab. Comput.} 14 (2005), 845--849.


\bibitem{Pik} O. Pikhurko,
A note on the Tur¨¢n function of even cycles,
\emph{Proc. Amer. Math. Soc.} 140 (2012), 3687--3692.


\bibitem{Sa} G. S\'ark\"ozy,
Cycles in bipartite graphs and an application in number theory,
\emph{J. Graph Theory} 19 (1995), 323--331.


\bibitem{SV08}
B. Sudakov and J. Verstra\"ete,
Cycle lengths in sparse graphs,
\emph{Combinatorica} 28 (2008), 357--372.


\bibitem{SV16}
B. Sudakov and J. Verstra\"ete,
The extremal function for cycles of length $l$ mod $k$,
\emph{arXiv:1606.08532}



\bibitem{Th83}
C. Thomassen, \newblock{Graph decomposition with applications to subdivisions and
path systems modulo k}, \newblock{\emph{J. Graph Theory}} \textbf{7} (1983), 261--271.

\bibitem{turan}
P. Tur\'an, Eine Extremalaufgabe aus der Graphentheorie,  \emph{Mat. Fiz. Lapook}
48 (1941), 436 -- 452.

%\bibitem{Th88}
%C. Thomassen, \newblock{Paths, circuits and subdivisions},
%\newblock{\emph{Selected Topics in Graph Theory (L. Beineke and R. Wilson, eds.)}}, vol. 3, Academic Press, 1988, pp. 97--131.



\bibitem{V00} J. Verstra\"ete,
On arithmetic progressions of cycle lengths in graphs,
\emph{Combin. Probab. Comput.} 9 (2000), 369--373.


\bibitem{V16} J. Verstra\"ete,
Extremal problems for cycles in graphs,
In \emph{Recent Trends in Combinatorics},
A. Beveridge et al. (eds.), The IMA Volumes in Mathematics and its Applications 159, 83--116, Springer, New York, 2016.




\end{thebibliography}
\end{document}